\documentclass[11pt,a4paper,leqno,twoside]{amsart}
\usepackage{amsmath,amssymb,amsthm,enumerate,hyperref,color}
\numberwithin{equation}{section}
\newtheorem{theorem}{Theorem}[section]

\newtheorem{proposition}[theorem]{Proposition}
\newtheorem{lemma}[theorem]{Lemma}

\newtheorem{remark}[theorem]{Remark}

\newcommand{\marginlabel}[1]{\mbox{}\marginpar{\raggedleft\hspace{0pt}\tiny#1}}

\newcommand{\nor}{\Arrowvert}

\def\up{u_p}

\def\e{{\varepsilon}}

\def\d{\delta}

\def\l{{\lambda}}
\def\a{{\alpha}}

\def\de{\partial}

\def\lra{\longrightarrow}

\def\cC{\mathcal{C}}
\newcommand{\R}{\mathbb{R}}
\newcommand{\N}{\mathbb{N}}

\newcommand{\B}{B}

\newcommand{\n}{\mathop{N}}

\newif\ifcomment \commentfalse
\def\commentON{\commenttrue}

\long\outer\def\BC#1\EC{\ifcomment \sloppy \par \# \ldots\dotfill
{\em #1} \dotfill \# \par \fi } \commentON

\newcommand{\remove}[1]{}

\title{Nonradial sign changing solutions to Lane Emden equation}
\author[A.~L.~Amadori, F.~Gladiali]{Anna Lisa Amadori$^\dag$ and Francesca Gladiali$^\ddag$}
\thanks{The authors are members of the Gruppo Nazionale per
l'Analisi Matematica, la Probabilit\'a e le loro Applicazioni (GNAMPA)
of the Istituto Nazionale di Alta Matematica (INdAM). The second author is supported by PRIN-2009-WRJ3W7 grant}
\date{\today}
\address{$\dag$ Dipartimento di Scienze Applicate, Universit\`a di Napoli ``Parthenope", Centro Direzionale di Napoli, Isola C4, 80143 Napoli, Italy. \texttt{annalisa.amadori@uniparthenope.it}}
\address{$\ddag$ Matematica e Fisica, Polcoming, Universit\`a di Sassari, via Piandanna 4, 07100 Sassari, Italy. \texttt{fgladiali@uniss.it}}

\begin{document}
\begin{abstract}
  In this paper we prove the existence of continua of nonradial solutions for the Lane-Emden equation. In a first result we show that there are infinitely many global continua detaching from the curve of radial solutions with any prescribed number of nodal zones.  Next, using the fixed point index in cone, we produce nonradial solutions with a new type of symmetry. This result also applies to solutions with fixed signed, showing that the set of solutions to the Lane Emden problem has a very rich and complex structure.
  
{\bf Keywords:} semilinear elliptic equations, nodal solutions, bifurcation.

{\bf AMS Subject Classifications:} 35J91, 35B05
\end{abstract}

\maketitle

\section{Introduction}
This work deals with the Lane-Emden problem 
\begin{equation} \label{eq1}
\left\{\begin{array}{ll}
-\Delta u = |u|^{p-1} u \qquad & \text{ in } A , \\
u= 0 & \text{ on } \partial A,
\end{array} \right.
\end{equation} 
where $A=\{x\in\R^{\n} \, : a<|x|<b\}$ is an annulus of $\R^{\n}$ with $\n\ge 2$ and $p>1$.
Problem \eqref{eq1} is said subcritical in dimension $2$ for any value of $p$ and in dimension $N\geq 3$ when $1<p<\frac{\n+2}{\n-2}$. In this case 
solutions can be found as critical points of the functional 
\begin{equation}\label{F}
F(u):=\frac 12\int_A |\nabla u|^2-\frac 1{p+1}\int_A |u|^{p+1}
\end{equation}
on the space $H^1_0(A)$ using the compact embedding of $L^{p+1}(A)$ into $H^1_0(A)$ for $p+1<\frac{2\n}{\n-2}$.

When $\n\geq 3$ and $p\geq \frac{\n+2}{\n-2}$ the problem is said critical or supercritical and variational methods or critical points theory cannot be used to find solutions because the functional $F$ is not well defined on the space $H^1_0(A)$ nor compact.
Anyway this issue is well settled in the framework of radial solutions
since the embedding of $L^{p+1}_{\mathit{rad}}(A)$ into  $ H^1_{0,\mathit{rad}}(A) $ is compact for every value $p$  and any dimension $\n$. This is indeed a consequence of the radial Lemma (see for example \cite{Ni}) and problem \eqref{eq1} always admits a positive solution in the space $ H^1_{0,\mathit{rad}}(A)$. Further sign changing solutions in $ H^1_{0,\mathit{rad}}(A)$ can be found using the minimization procedure introduced in \cite{BW}, see also \cite{AGG16} where this procedure is applied to problem \eqref{eq1}.

In particular for any $m\in\N$, there are exactly two radial solutions to \eqref{eq1} which have exactly $m$ nodal zones, and they are one the opposite of the other, see \cite{NN}.

So for any $p>1$ and $m\in \N$ there is only one radial solution to \eqref{eq1} which has $m$ nodal zones and is positive in the first one (i.e. has positive derivative in $|x|=a$): we shall denote it by $u^m_p$ and introduce the curve
\[
\mathcal{S}^m:=\big\{  (p,u^m_p)\ : \ p\in(1,+\infty) \big\}
.\] 
In the following we shall consider the number of nodal zones $m$ fixed, therefore we shall omit the dependence by $m$, when it does not give rise ambiguity. 
\\
In this paper we find nonradial solutions which spread from $\mathcal S$ by using the bifurcation theory: a powerful tool for dealing with supercritical problems, which has already been applied to solutions with fixed-sign in \cite{GGPS11}. 
To this aim a necessary condition  is a change of the Morse index of $u_p$. 
The computation of the Morse index of nodal radial solutions to the Lane Emden problem has been object of various studies, recently, for instance in the unit ball it has been exactly computed for large values of $p$ in \cite{dMIP16} when $\n=2$ and for almost critical values of $p$ in \cite{dMIP} when $\n\ge 3$.
About annular type domains, the paper \cite{AGG16} gives a characterization of the Morse index in terms of a related Sturm-Liouville problem and describes its asymptotic behaviour as $p\to 1$ and $+\infty$. This yields, incidentally, that there are infinitely many values of $p$ where the Morse index does change, although  there still are nontrivial difficulties in deducing actual bifurcation because there is no evidence that such change is odd.  Since our problem can be supercritical no variational structure can be used and only an {\em{odd}} change in the Morse index can produce a bifurcation result.
Such difficulties  do not present when dealing with positive solutions, because in that case  only the first eigenvalue of a radial associated problem can produce degeneracy
and this assures that the related eigenfunctions of the Laplace-Beltrami operator which are $O(\n-1)$-invariant span a  one dimensional space. This argument is the key for the closing part of the proof of \cite[Theorem 1.7]{GGPS11} and applies also to more general (non-autonomous) problems, as in \cite{GGN13}, \cite{AG}, or \cite{AG16}. 
With this approach one can construct nonradial solutions which still are $O(\n-1)$-invariant.
Unfortunately it does not apply to nodal solutions, because the structure of degeneracy is much more involved and it is needed to take into account the eigenfunctions of the Laplace-Beltrami operator related to many different eigenvalues: in general this could produce an even change in the Morse index.
In the present paper we overcome this obstacle and present two slightly different bifurcation results.
In the first one bifurcation is obtained by the classical Leray-Shauder degree approach, after  showing that the change of the Morse index is actually odd or in a suitable linear space which varies according to $p$ and to the dimension $\n$. It can be the entire $C^{1,\gamma}_0(A)$ either or a subspace of type
\begin{equation}\label{X}
  X^n : =  \left\{u\in     C^{1,\gamma}_0(A)  \, : \, u(r,\varphi,\theta) \hbox{ is } {2\pi}/n  \hbox{ periodic  and even w.r.t. }\varphi\right\}.
\end{equation}  
Here we have denoted by $(r,\varphi,\theta)$  the spherical coordinates in $\R^{\n}$  with $r=|x|\in[0,+\infty)$, $\varphi\in[0,2\pi]$ and $\theta =(\theta_1,\dots\theta_{\n\!-\!2})\in (0,\pi)^{\n-2}$:
  \begin{equation}\label{sphericalcoord}\begin{array}{lr}
      x_1=r \cos\varphi \prod\limits_{k=1}^{\n\!-\!2}\sin\theta_k  , & \\
      x_2= r \sin\varphi \prod\limits_{k=1}^{\n\!-\!2}\sin\theta_k  , & \\
      x_{k} = r \cos \theta_{k\!-\!2} \prod\limits_{h=k-1}^{\n\!-\!2}\sin\theta_h &  \qquad \mbox{ as } k=3,\dots \n-1 , \\
      x_{\n}= r \cos \theta_{\n\!-\!2}. & \end{array}
    \end{equation}
In doing so we give a detailed description of which eigenfunctions of the Laplace-Beltrami operator come into play, which can turn useful also for other applications. It has to be emphasized that the method used here applies to any kind of solution (with fixed sign or sign-changing) and can be used also with other supercritical nonlinear terms: for instance in  \cite{AGnodalHenon} it is applied to the non-autonomous nonlinearity arising in the Henon equation.
The precise terms state as follows
\begin{theorem}\label{teo:bifurcation}
  Let $m\ge 2$. There is a sequence $p_n\to+\infty$ such that a nonradial bifurcation occurs at $(p_n,u^m_{p_n})$. The bifurcation is global and the Rabinowitz alternative holds.
  \end{theorem}
The claim holds also for $m=1$, but it has already been proved before. 
The approach used here does not provide further information about the symmetry of the nonradial bifurcating solutions that do not necessarily belong to $X^n$ (apart from $\n=2$) and  does not allow to separate continua one from another. 
This issue can be tackled by turning to the notion of degree and index of fixed points in cones introduced by Dancer in \cite{D83}. This quite technical notion of index has been used by Dancer in \cite{D92} to show separation of global branches of positive solutions in an annulus in 2 dimensions.
Indeed in that case the branches of bifurcating solutions which belong to different cones are separated and, thanks to positiveness and subcriticality, they exist for every value of the bifurcating parameter, giving also a multiplicity result.
Dancer said in \cite{D92} that his result was rather two dimensional and suggested to generalize to the higher dimensional case using the $O(\n-1)$-invariant spherical harmonics which depend only on one angle in spherical coordinates. Indeed this effort was done in \cite{G16} where the author managed to separate the first two branches of bifurcating solutions from the others.
The main difficulty in this last paper was to understand which kind of symmetries were inherited by nonradial solutions, since the symmetries of the $O(\n-1)$-invariant spherical harmonics were unclear. 
Here we continue the study in \cite{G16} in force of the analysis performed in the proof of Theorem \ref{teo:bifurcation}. Indeed in dealing with sign changing solutions a careful study of the spherical harmonics was necessary to understand and overcome the degeneracy. Finally it comes out that the $O(\n-1)$-invariant spherical harmonics were not the good ones to generalize the result of \cite{D92} and the right symmetry and periodicity properties are the one with respect to the last angle $\varphi$ considered in the space $X^n$.
Summing up, the suitable cones are
\begin{equation}\label{K^n}\begin{array}{rl}
K^n:=\{v \in X^n  \, : \, & v(r,\varphi,\theta) \hbox{ is nonincreasing w.r.t. } \varphi \\
& \hbox{ for all } (r,\varphi,\theta)\in (a,b)\times [0,{\pi}/n)\times [0,\pi]^{\n\!-\!2} \}.
\end{array}\end{equation}

Actually  there is a continuum of nonradial solutions  which belong to the cone $K^n$, (for any $n$ sufficiently large or any $n$ when $m=1$) and that it is not compact in $(1,+\infty)\times C^{1,\gamma}_0(A)$.
Let 
$\Sigma_n^m$ be the closure of the set
\[ \{ (p, u) \in (1,+\infty)\times K^n\setminus \mathcal{S}^m  \, : \, u \mbox{ solves \eqref{eq1}} \}. 
\]
We prove that

\begin{theorem}\label{teo:bifurcationincones}
  For all $m\ge 1$ there exists 
$\bar{n}=\bar{n}(m)\in\N$  
and a sequence $\left(p_n, u^m_{p_n}\right)_{n\ge\bar{n}}$ such that  $(p_n,u^m_{p_n})\in \Sigma^m_n$ and $p_n\to+\infty$.
  Moreover letting $C_n$ be the closed connected component of $\Sigma^m_n$ that contains $(p_n,u^m_{p_n})$, these continua are global and the following Rabinowitz alternative holds, namely every $C_n$ contains a sequence $(p_k,u_k)$ such that
\begin{enumerate}[i)]
\item or $p_k\to +\infty$,
\item or $\|u_k\|_{C^{1,\gamma}(A)}\to +\infty$,
\item or, for $m>1$, $p_k\to 1$, $\|u_k\|_{C^{1,\gamma}(A)}<C$ and $u_k$  converges to an eigenfunction of the Laplacian.
\end{enumerate}
\end{theorem}
Since we are dealing with sign changing solutions and with problems that are, in many cases, supercritical we are not able to follow the continua and to say something about the interval of parameter {$p$} in which this different nonradial sign changing solutions exist.
Another issue is the separation of continua. In dimension $N=2$, the intersection of two cones $K_n$ and $K_i$ with $n\neq i$ consists in radial solutions. This allows to separate branches  of positive solutions, while a continuum of nodal solutions could in principle link to the curve of radial solutions with a different number of nodal zones. Passing to higher dimension $N\ge 3$, the intersection of two cones includes also  functions which are not radial, but do not depend by the angle $\varphi$.
A numerical study of these solutions will be of help to understand their behaviour.

Item {\it iii)}, i.e.~the fact that nonradial solutions can possibly accumulate near at eigenfunctions of the Laplace operator, is peculiar of nodal solutions ($m>1$). We do not know if case {\it iii)} actually occurs, but we are not able to rule it out because many sign changing nonradial solutions exist for $p$ close to $1$ even if one prescribes the number of nodal domains.
As an example in $\R^{\n}$ we can find at least $\n+1$ solutions with $2$ nodal domains, $\n$ minimizing the functional $F$ in \eqref{F} in the space of functions which are {\em{odd}} with respect to some direction $x_i$ and one radial. Further we cannot exclude the fact that the number of nodal domains increases along the {continuum}.

Coming back to positive solutions, Theorem \ref{teo:bifurcationincones} brings forth a deeper understanding.
Notice that the functions in $K^n$ are not $O(\n-1)$ invariant and therefore these solutions do not coincides with the positive ones produced in \cite{GGPS11}. Moreover the { continua }
do not coincide with the two found in \cite{G16}. In particular in dimension $N\geq 3$ for $n=1$ and $n=2$ we obtain at least two nonradial different continua bifurcating from the same points: the ones in \cite{G16} and the ones given in Theorem \ref{teo:bifurcationincones}.
On the other hand, positive solutions cannot develop nodal zones, and item  {\it iii)}  is excluded since a  uniqueness result holds when $p$ close to $1$.
For the reader convenience we restate the bifurcation result for positive solutions.
\begin{theorem}\label{teo:bifurcationinconespos}
 Let $u_p$ be the positive solution to \eqref{eq1}. For all $n\in\N$ there is a 
{continuum}
of positive nonradial solutions in $K^n$, and a sequence $(p_n,u_{p_n}) \in \Sigma^1_n$.
 Moreover $C_n$, the closed connected component of $\Sigma^1_n$ that contains $(p_n,u_{p_n})$, is unbounded.
\end{theorem}
This paper is organized at follows: in Section \ref{S2} we recall the results about nondegeneracy and Morse index obtained in \cite{AGG16} and deduce Theorem \ref{teo:bifurcation}, while in Section \ref{S3} we prove the global bifurcation stated by Theorems \ref{teo:bifurcationincones} and \ref{teo:bifurcationinconespos}.

\section{A bifurcation result from radial nodal solutions}\label{S2}

The starting point to show that a branch of nonradial solutions bifurcates from the radial one is degeneracy, i.e.~existence of a nontrivial solution of the linearized equation
\begin{equation}\label{eq:lin}
\left\{\begin{array}{ll}
-\Delta w = p |u_p|^{p-1} w & \qquad \text{in } A, \\
w=0 & \qquad \text{ on } \partial A.
\end{array}\right.
\end{equation}
Another point is a change of the Morse index, i.e.~the number (with multiplicity) of negative eigenvalues for the linearized problem, i.e.
\begin{equation}\label{eq:eigenvalue}
\left\{\begin{array}{ll}
-\Delta w = p |u_p|^{p-1} w+\mu w & \qquad \text{in } A, \\
w=0 & \qquad \text{ on } \partial A.
\end{array}\right.
\end{equation}
\cite[Lemma 2.1]{GGPS11}  remarks that problem \eqref{eq:eigenvalue} is in some sense equivalent to the weighted eigenvalue problem 
\begin{equation}\label{eq:eigenvaluew}
\left\{\begin{array}{ll}
-\Delta w = p |u_p|^{p-1} w+ \frac{\tilde\mu}{|x|^2} w & \qquad \text{in } A, \\
w=0 & \qquad \text{ on } \partial A.
\end{array}\right.
\end{equation}
Precisely the Morse index of $u_p$ is the number of negative eigenvalues of \eqref{eq:eigenvaluew}, each counted with its multiplicity.
This remark allows, by projecting \eqref{eq:eigenvaluew} along spherical harmonics, to  relate both items of degeneracy and Morse index  to a Sturm-Liouville problem 
\begin{align}\label{eq:linauto}
\left\{\begin{array}{ll}
-\phi''- \dfrac{\n-1}{r} \phi'= \left(p |u_p|^{p-1} +\dfrac{\nu}{r^2} \right)\phi \qquad a<r<b , \\
\phi(a)=\phi(b)=0, \end{array}\right. 
\end{align}
This item has been exploited in a previous work \cite{AGG16}, extending analogous properties already remarked in \cite{GGPS11} for positive solutions.
We recall here the main results of that paper which are useful to the present purpose.
Henceforth we write $\nu_i(p)$ for the sequence of the eigenvalues for problem \eqref{eq:linauto}, 
\[ \lambda_j=j(\n-2+j) \]
for the $j^{th}$ eigenvalue of the Laplace-Beltrami operator in ${\mathbb S}^{\n-1}$ and
\begin{equation}\label{dim-arm}
 N_j= \frac{(N+2j-2)(N+j-3)!}{(N-2)! j!}
\end{equation}
for its multiplicity.

\begin{proposition}\label{degeneracy}
  For every $p>1$ and $m\ge 1$, the nodal radial solutions $u_p^m$ to  \eqref{eq1} are radially nondegenerate, and degenerate if and only if 
\begin{align}\label{il}
  \nu_i(p)+\lambda_j=0  \quad  \mbox{as $i=1,\cdots m$ and $j\ge 2$, or  $i=m$ and $j= 1$.}
\end{align}
The values of $p$ such that $u^m_p$ is degenerate are isolated and build an increasing diverging sequence.  
Moreover the Morse index of $u^m_p$ is given by the formula
\begin{equation}\label{morseformula}
    m(u^m_p)=\sum_{i=1}^{m}\sum _{j<J_i(p)} N_j,
    \end{equation}
where $J_i(p)=\left(\sqrt{(\n-2)^2-4\nu_i(p)}-\n+2\right)/2$
and
  \begin{align}\label{asymptoticnuinfty}
    \lim\limits_{p\to+\infty}\nu_i(p)=& -\infty \qquad && \mbox{for all } i=1,\dots m, \\
     \label{asymptoticnu1l}
     0\ge\limsup\limits_{p\to 1^+}\nu_i(p)\ge \liminf\limits_{p\to 1^+} \nu_i(p) > & -\infty  \;  && \mbox{for } i=1,\dots m-1, \\
     \label{asymptoticnu1m}
     \lim\limits_{p\to 1^+}\nu_m(p)=& 0.  && 
  \end{align}
 \end{proposition}

Even though the information collected in Proposition \ref{degeneracy}  imply that there are points where $u_p$ is degenerate and the Morse index actually changes, there still are nontrivial difficulties in deducing global bifurcation because there is no evidence that such change is odd.
Such difficulties  do not present when dealing with positive solutions, because in that case it is clear by \eqref{il}  that only the first eigenvalue $\nu_1(p)$ can produce degeneracy. So one can 
remark that the related eigenfunctions of the Laplace-Beltrami operator which are $O(\n-1)$-invariant span a  one dimensional space getting a global bifurcation result. 
Unfortunately it does not apply to nodal solutions, because it can happen that $\nu_i(p)+\lambda_j=0$ for many values of $i$ and $j$, at the same $p$. Therefore it is needed to take into account the eigenfunctions of the Laplace-Beltrami operator related to many different eigenvalues, and in general this could produce an even change in the Morse index.
In the following we show that the change of the Morse index is actually odd in a suitable linear space, which can be the entire $C^{1,\gamma}_0(A)$ or $X^n$ introduced in \eqref{X} (for some properly chosen $n$) and this gives the global bifurcation result.
To enter the details it is of help introducing the operator
\begin{align}
  T(p,v) : & \, (1,+\infty)\times  C^{1,\gamma}_0(A)\lra  C^{1,\gamma}_0(A) , \nonumber  \\ \label{T}
 & T(p,v)=(-\Delta)^{-1}\left(|v|^{p-1}v\right).
\end{align}
It is clear that $(p,u_p)$ is a curve of fixed points for $T$ and more generally $u$ solves \eqref{eq1} when $u=T(p,u)$.
Coming to the linear space $X^n$ defined in \ref{X}, we have that $u_p \in X^n$ for all $n$ and $T(p,\cdot)$ maps  the space $X^n$ into itself.

\begin{lemma}\label{lemmaTX}
  Let $p>1$, $n\in\N$ and $v\in X^n$. Then $T(p,v)\in X^n$.
 \end{lemma}
\begin{proof}
Let $v\in X^n$ and let $z:=T(p,v)$. Then $z$ is a weak solution to $-\Delta z=|v|^{p-1}v$ in $A$ with Dirichlet boundary conditions. Letting $\tilde z:=z(r,\varphi+\frac{2\pi}n,\theta)$ it is easy to see that $\tilde z$ weakly solves $-\Delta \tilde z=|v(r,\varphi+\frac{2\pi}n,\theta)|^{p-1}v(r,\varphi+\frac {2\pi}n,\theta)$ in $A$ with Dirichlet boundary conditions, and since $v\in X^n$ then $-\Delta \tilde z=|v(r,\varphi,\theta)|^{p-1}v(r,\varphi,\theta)=|v|^{p-1}v$ in $A$. Then $\tilde z$ and $z$ satisfy the same equation in $A$ and then $\tilde z=z$. This implies that $z\in X^n$ and concludes the proof.
\end{proof}

We are now in the position to prove Theorem \ref{teo:bifurcation}.

\begin{proof}[Proof of Theorem \ref{teo:bifurcation}]
 In force of \eqref{asymptoticnuinfty} and \eqref{asymptoticnu1l}, for any $m\geq 2$ there exists $\bar n(m)$ such that the quantity $\nu_1(p) + \lambda_n $ changes sign for $p\in (1,+\infty)$ and $n\geq \bar n$. In particular 
for all sufficiently large $n\in\N$ there exist at least one $p_n \in (1,+\infty)$ such that
\begin{align}\label{degeneracypoint}
  \nu_1(p_n) + \lambda_n =0 , \qquad  \left(\nu_1(p_n-\delta) +\lambda_n \right) \left(\nu_1(p_n+\delta) +\lambda_n \right) < 0
\end{align}
for suitable $\delta>0$.
Since the degeneracy points are isolated by Lemma \ref{lemmaTX} we can also take that 
\begin{equation}\label{i}
  \nu_i(p) \neq -\lambda_j \; \mbox{ for all $i , j \in \N$, \, if $p\neq p_n$}
\end{equation}
for $p\in[p_n -\delta, p_n + \delta]$.

  Assume by contradiction that $(p_n, u_{p_n})$ is not a bifurcation point in $W$ (for $W=C^{1,\gamma}_0(A)$ or $X^n$). Then there exists $\delta >0$ so that properties \eqref{degeneracypoint}, \eqref{i} hold, and neighborhoods ${\mathcal O}_p$ of $u_{p}$ in $W$ such that $u-T(p,u)\neq 0$ for every $(p,u)\neq (p,u_p)$ in the boundary of the set $\left\{ (q,v) \, : \, q\in (p_n-\delta, p_n+\delta) , \, v\in {\mathcal O}_q\right\}$.
Hence
\[ {\mathrm{deg}}_W(I-T(p_n-\d,\cdot),{\mathcal O}_{p_n-\d},0)={\mathrm{deg}}_W(I-T(p_n+\d,\cdot),{\mathcal O}_{p_n+\d},0).\]
But by \eqref{i} we know that ${\mathrm{deg}}_W(I-T(p_n\pm\d,\cdot),{\mathcal O}_{p_n\pm\d},0)=(-1)^{m_W(p_n\pm \d)}$,
so that $m_W(p_n +\d) - m_W(p_n -\d)$ must be an even number.
Here $m_W(p)$ stands for the Morse index of  $u_p$ in $W$, i.e. the number of negative eigenvalues of \eqref{eq:eigenvalue}, each multiplied by the number of respective eigenfunctions contained in $W$.
On the other hand $m_W(p)$ is equal to the number of linearly independent eigenfunctions of \eqref{eq:eigenvaluew} in $W$ corresponding to negative eigenvalues.
It is easily seen  that the eigenvalues of \eqref{eq:eigenvaluew} are $\tilde\mu = \nu_i(p)+\lambda_j$ and the related eigenfunctions  are of type 
\begin{equation}\label{eigenfunction_formula}
  w_{i j}(r,\varphi,\theta) = \phi_i(r) \, Y_j(\varphi,\theta),
  \end{equation}
where $\phi_i$ is an eigenfunction of \eqref{eq:linauto} related to $\nu_i(p)$ and $Y_j$ is an eigenfunction
of the Laplace-Beltrami operator on ${\mathbb S}^{\n\!-\!1}$ related to the eigenvalue $\lambda_j$. \\
Summing up the eigenfunctions of type $w_{1,n}$ cause  a change in the Morse index from $p_n-\delta$ to $p_n+\delta$,
  but there can be an additional change due to other eigenfunctions of type $w_{i j}$ because the degeneration in $p_n$ can have a quite involved structure.
  Remembering \eqref{i} and the fact that $\nu_1(p)<\nu_i(p)$, the other indices $i,j$ that come into play are that ones satisfying
  \begin{align}\label{ii}
   \begin{array}{c} \nu_i(p_n ) + \lambda_j=0 \mbox{ for some $i\in\{2,\dots,m\}$ and } j\in\{1,\dots  n -1\}, \\
     (\nu_i(p_n-\delta ) + \lambda_j) (\nu_i(p_n+\delta ) + \lambda_j) <0 . \end{array}
  \end{align}
  More precisely the Morse index increases if $\nu_i(p_n-\delta)+\lambda_j>0$ and it decreases in the opposite case.
  The amount of such change (in absolute value) is equal to $N_j(W)$, that is the multiplicity of $\lambda_j$ as an eigenfunction of the Laplace-Beltrami operator in the linear space $W$.
  We can sum up these comments in the formula
  \begin{equation}\label{formula}
  m_W(p_n +\d) - m_W(p_n -\d)  =  \sum\limits_{j=0}^{n} \chi_j N_j(W) ,
  \end{equation}
where  $\chi_n= {\mathrm{sign}}\left(\nu_1(p_n-\delta)+\lambda_n\right)$ and 
  \[ \chi_j = \begin{cases} {\mathrm{sign}}\left(\nu_i(p_n-\delta)+\lambda_j\right)  & \mbox{ if there is $i=2\dots m$ satisfying \eqref{ii}} \\
      0 & \mbox{ otherwise} \end{cases}\]
  as $j=1\dots n-1$. 

  Eventually if $(p_n, u_{p_n})$ is not a bifurcation point in $W$ then $\sum\limits_{j=0}^{n} \chi_j N_j(W)$ is even. 
    Otherwise, i.e.~when $\sum\limits_{j=0}^{n}  \chi_j  N_j(W)$   is odd,
    we have a contradiction and $(p_n,u_{p_n})$ is  a bifurcation point. Moreover following \cite[Theorem 3.3]{G} one sees that the bifurcation is global and the Rabinowitz alternative holds.

\subsection*{Case $\n=2$}
In this case the eigenfunctions of the Laplace-Beltrami operator in ${\mathbb S}^1$ related to the eigenvalue $\lambda_j=j^2$ are of the form
\[ Y_j(\varphi)=A \cos(j\varphi)+ B \sin(j\varphi), \]
and they build a space of dimension 2 in $C^{1,\gamma}_0(A)$ for any $j\geq 1$.
But $Y_j\in X^n$ only if $B=0$ and $j$ is a multiple of $n$, hence $N_n(X^n)=1$ and $N_j(X^n)=0$ for $j<n$, so that \eqref{formula} gives 
\[ m_{X^n}(p_n +\d) - m_{X^n}(p_n -\d)  =  \chi_n = \pm 1 \]
and we have reached a contradiction. Therefore bifurcation actually occurs in the space $X^n$ and it is global.

\subsection*{Case $\n=3$}
Now the eigenfunctions of the Laplace-Beltrami operator in ${\mathbb S}^2$ related to the eigenfunction $\lambda_j=j(1+j)$ are of type
\begin{equation}\label{N=3}
Y_j(\varphi,\theta)=\sum\limits_{\ell=0}^j  P_j^{\ell}(\cos\theta)\left(A_{\ell} \cos \ell\varphi+ B_{\ell}\sin \ell\varphi\right), 
\end{equation}
where $P_j^{\ell}$ are the associated Legendre polynomials.
The eigenfunctions build a subspace of $C^{1,\gamma}_0(A)$ of (odd) dimension $N_j=1+2j$ (see \eqref{dim-arm}).
If there is not any index $i=2,\dots m$ satisfying \eqref{ii}, then \eqref{formula} gives
\[ 
  m(p_n +\d) - m(p_n -\d)  = \chi_n N_n = \pm (1+2n) ,
  \]
which is odd.
The same conclusion holds true if  
\[\sum\limits_{j=1}^{n-1} \chi_j N_j= \sum\limits_{j=1}^{n-1} \chi_j (1+2j)\]
is even. In these cases bifurcation occurs in $C^{1,\gamma}_0(A)$ and no further information about the symmetries of the nonradial bifurcating solutions are provided.

Otherwise we turn to $X^n$.
Now for $j<n$ $Y_{j}(\varphi,\theta)\in X^n$ only if $B_l=0$ for $l=0,1,\dots  j $ and $A_l=0$ for $l=1,\dots  j $ by \eqref{N=3} (i.e.~there is only 1 linearly independent eigenfunction, corresponding to $A_0$). So the contribution to the Morse index of such indices is given by
\[\sum\limits_{j=1}^{n-1} \chi_j N_j(X^n)= \sum\limits_{j=1}^{n-1} \chi_j, \]
which has the same parity of $ \sum\limits_{j=1}^{n-1} \chi_j (1+2j) $ and is therefore odd, in this case.
Besides $N_n(X^n)=2$ because $Y_{n}(\varphi,\theta)\in X^n$ only if $B_l=0$ for $l=0,1,\dots  n $ and $A_l=0$ for $l=1,\dots  n -1$ by \eqref{N=3} (i.e.~there are 2 linearly independent eigenfunctions, corresponding to $A_0$ and $A_n$).
Therefore by \eqref{formula} the overall change of the Morse index in $X^n$ is 
\[m_{X^n}(p_n +\d) - m_{X^n}(p_n -\d)  =\sum\limits_{j=1}^{n} \chi_j N_j(X^n)= 2\chi_n+\sum\limits_{j=1}^{n-1} \chi_j, \]
that is an odd number, and bifurcation occurs in $X^n$.
\subsection*{General case $\n\in\N$, $\n\ge 4$}
Now the eigenfunctions of the Laplace-Beltrami operator in ${\mathbb S}^{\n-1}$ related to the eigenvalue $\lambda_j=j(\n-2+j)$ are of type
\[Y_j(\varphi,\theta)=\mathop{\sum\limits_{\ell=0}^{j} \prod\limits_{k=1}^{\n\!-\!2}}\limits_{\substack{i_0\le i_1\dots\le i_{\n\!-\!2}\\ i_0=\ell, \; i_{\n\!-\!2} =j}} G_{i_{k}}^{i_{k\!-\!1}}(\cos\theta_k,k\!-\!1) \left(A_{\ell}^{i_1\dots i_{\n\!-\!3}} \cos \ell\varphi+ B_{\ell}^{i_1\dots i_{\n\!-\!3}}\sin \ell\varphi\right), \]
where $G_{i}^{0}(\cdot,k)$ are the Gegenbauer polynomials and  are generated by a Taylor expansion
\begin{align*}
\sum\limits_{i=0}^{\infty}G_{i}^0(\omega,k) x^i = (1-2x\omega+x^2)^{-(1+k)/2} 
\end{align*}
and
\begin{align*}
 G_{i}^{\ell}(\omega,k)  & =  (1-\omega^2)^{\frac{\ell}{2}} \frac{d^{\ell}}{d\omega^{\ell}}G_{i}^0(\omega,k) .
\end{align*}
Moreover  $G_{i}^{\ell}(\omega,0) =  P_{i}^{\ell}(\omega) $ where as before 
$P_j^{\ell}$ are the associated Legendre polynomials.
We refer to the note \cite{W} for more details about the representation of spherical harmonics functions by means of the Gegenbauer polynomials.

The eigenspace has dimension $N_j=\binom{\n+j-1}{\n-1}-\binom{\n+j-3}{\n-1}$ in $C^{1,\gamma}_0(A)$.

If  $j< n $, only the eigenfunctions not depending by $\varphi$ belong to $X^n$, which gives $N_j(X^n)=\binom{\n+j-3}{\n-3}$ linearly independent eigenfunctions (corresponding to  $A_0^{i_1\dots i_{\n\!-\!3}}$, where $i_1\dots i_{\n\!-\!3}$ is a ($\n-3$)-combination of  integers among $0,\dots j$, with possible repetitions).
We claim that if the dimension  of the eigenspace in $C^{1,\gamma}_0(A)$ is even/odd, the same holds in $X^n$, equivalently $N_j-N_j(X^n)$ is even.
But this is clear because $N_j-N_j(X^n)=  2\binom{\n+j-3}{\n-2}$.
Hence the eventual contribution to the change of the Morse index coming from the eigenvalues satisfying \eqref{ii}, that is $\sum\limits_{j=0}^{n-1}\chi_j N_j(W)$, is simultaneously even or odd in both spaces $W=C^{1,\gamma}(A)$ and $W=X^n$. 
Besides for $j= n $ there are $N_n(X^n)=1+\binom{\n+ n -3}{\n-3}$ linearly independent eigenfunctions in ${X}(n)$, because  the only eigenfunction containing $\cos  n \varphi$ adds to the ones considered for $j< n $. It follows that $N_n(X^n)$ is even when $N_n$ is odd, and viceversa.
Summing up, the overall change of the Morse index given by \eqref{formula} is odd in the whole space $C^{1,\gamma}_0(A)$ or in the subspace  $X^n$ and the proof is completed.
\end{proof}

\section{Solutions inside the cones $K^n$}\label{S3}

Theorem \ref{teo:bifurcation} shows that there is an infinite number of  parameters $p_n$ such that $(p_n, u_{p_n})$ is a bifurcating point, i.e.~in any neighborhood of $(p_n,u_{p_n})$ there are points $(p,u)$ where $u$ is not radial and $T(p,u)=u$.
This result is derived by the theory of Leray-Shauder topological degree. 
With the aim of studying continuum of propagating nonradial solutions, we give an alternative proof of Theorem \ref{teo:bifurcation} which makes use of the notion of index of fixed points in cones introduced by Dancer in \cite{D83}.
In any dimension $\n$ we look into cones of type
\[\begin{array}{rl}
K^n:=\{v \in X^n  \, : \, & v(r,\varphi,\theta) \hbox{ is nonincreasing w.r.t. } \varphi \\
& \hbox{ for all } (r,\varphi,\theta)\in (a,b)\times [0,{\pi}/n)\times [0,\pi]^{\n\!-\!2} \},
\end{array}\]
which were already introduced in \eqref{K^n}.

It is clear that $u_p\in K^n$ for any $p$ and $n$ since it does not depend on $\varphi$. Let us 
 check that $T(p,\cdot)$ maps $K^n$ into itself.

\begin{lemma} 
If $g\in K^n$, then $z=T(p,g)\in K^n$.
\end{lemma}
  \begin{proof}
    $z=T(p,g)$ solves
    \[ \begin{cases}
-\Delta z=|g|^{p-1}g & \hbox{ in }A \\
z=0 & \hbox{ on }\partial A
    \end{cases}\]
    We have already noticed that $z\in X^n$, it is left to show that $z(r,\varphi,\theta)$ is nonincreasing w.r.t. $\varphi\in(0,2\pi/n)$. So we look at the Laplace equation in radial coordinates and we check that $\zeta=\partial_{\varphi}z\le 0$ for $\varphi\in(0,\pi/n)$.

    If $\n=2$ we have
    \[\begin{cases}
    -\partial^2_{rr}z-\frac 1r \partial_r z-\frac 1{r^2}\partial^2_{\varphi\varphi} z=|g|^{p-1} g & \hbox{ as } (r,\varphi)\in (a,b)\times [0,2\pi)\\
      z(a,\varphi)=z(b,\varphi)=0 &  \forall \varphi\in [0,2\pi),
    \end{cases}\]
    and then
    \[\begin{cases}
    -\partial^2_{rr} \zeta-\frac 1r \partial_r \zeta-\frac 1{r^2}\partial^2_{\varphi\varphi} \zeta =p|g|^{p-1} g_\varphi  & \hbox{ as } (r,\varphi)\in (a,b)\times(0,\pi/n) , \\
    \zeta(a,\varphi)=\zeta(b,\varphi)=0 & \hbox{ for all } \varphi\in (0, {\pi}/n), \end{cases}\]
    with $g_{\varphi}=\partial_{\varphi}g\le 0$ in $(a,b)\times(0,\pi/n)$ because $g\in K^n$.
    Moreover $\zeta(r,0)=0$ because $z\in X^n$ is even w.r.t.~$\varphi$, and also $\zeta(r,\pi/n)=0$ because $z$ is even and $2\pi/n$ periodic, so that $z(r,{\pi}/n+ \varphi)=z(r,{\pi}/n- \varphi)$.\\
    Summing up
    \[\begin{cases}
    -\Delta\zeta \le 0 & \mbox{ in } (a,b)\times (0,\pi/n) , \\
    \zeta=0 & \mbox{ on } \partial\left((a,b)\times (0,\pi/n)\right),
    \end{cases}\]
    and the maximum principle yields that $\zeta\le 0$ in $(a,b)\times (0,{\pi}/n)$.
    
 For a generic $\n\ge 3$, the Laplace operator can be written in the hyperspherical coordinates $(r,\varphi,{\theta})$ as
    \[\begin{array}{c}
      \Delta = \dfrac{1}{r^{\n\!-\!1}}\partial_r\left(r^{\n-1}\partial_r z\right)  + \dfrac{1}{r^2}\left( c_0({\theta}) \partial^2_{\varphi\varphi} +  \sum\limits_{k=1}^{\n-2} c_k({\theta})
\partial_{\theta_k}\left(\sin^k\theta_k\partial_{\theta_k}\right)\right) , \\
       c_0({\theta}) =  \prod\limits_{k=1}^{\n-2} \dfrac{1}{\sin^2\theta_k}, \qquad \qquad c_k({\theta}) = \dfrac{1}{\sin^k\theta_k}\prod\limits_{h=k+1}^{\n-2} \dfrac{1}{\sin^2\theta_h} .
    \end{array}\]
    So it is easily seen that $\zeta=\partial_{\varphi} z$ fulfills
    \[\begin{cases}
    -\Delta \zeta =p|g|^{p-1} \partial_{\varphi}g  & \hbox{ in } A , \\
    \zeta(a,\varphi,\theta)=\zeta(b,\varphi,\theta)=0 & \hbox{ for all } (\varphi,\theta)\in (0,2\pi)\times(0,\pi)^{\n\!-\!2} ,
    \end{cases}\]
    with $\partial_{\varphi}g\le 0$ in $A_n=(a,b)\times(0,\pi/n)\times(0,\pi)^{\n\!-\!2}$ because $g\in K^n$.
    Moreover $\zeta(r,0,\theta)=0$ because $z\in X^n$ is even w.r.t.~$\varphi$, and also $\zeta(r,\varphi/n,\theta)=0$ because $z$ is even and $2\pi/n$ periodic, so that $z(r,{\pi}/n+ \varphi,\theta)=z(r,{\pi}/n- \varphi,\theta)$.
    Lastly if $\theta\in \partial\left((0,\pi)^{\n\!-\!2}\right)$, then $(r,\varphi,\theta)$ belongs to the $\n\!-\!2$-hyperplane of Cartesian equations $x_1=x_2=0$ (see \eqref{sphericalcoord}). Then actually the point $(r,\varphi,\theta)=(r,0,\theta)$ does not change with $\varphi$, so that $z(r,\varphi,\theta)=z(r,0,\theta)$ for all $\varphi$ implying that $\zeta(r,\varphi,\theta)=0$.
    \\
   Eventually we have
    \[\begin{cases}
    -\Delta\zeta \le 0 & \mbox{ in } A_n, \\
    \zeta=0 & \mbox{ on } \partial A_n,
    \end{cases}\]
    and the maximum principle yields again that $\zeta\le 0$ in $A_n$, which ends the proof.
  \end{proof}
  
  \begin{remark}\label{phibuono}
      Among spherical coordinates, the angle $\varphi$ actually plays a special role when interested in Laplace operator, because it is the only variable with the property $\partial_{\varphi}\Delta u = \Delta \partial_{\varphi} u$.
 Extension to other nonlinear terms can be tried whenever they are monotone w.r.t.~$u$, namely $\partial_u f(u)\ge 0$. Indeed in that case the operator $T(u)= \Delta^{-1}f(u)$ preserves the cones of type $K^n$.
  \end{remark}
  
Our aim is to show that  a continuum of nonradial solutions contained in $K^n$ moves away from the branch of radial solutions and this continuum is not compactly contained in $(1,+\infty)\times C^{1,\gamma}(A)$.
A crucial role is played by  an interesting property of the index of fixed point in cones in \cite{D83}: naive intuition expects that at each value of $p$ such that $\nu_i(p)+\lambda_n$ changes sign, then the index changes, but this is the case only for $i=1$, indeed. Roughly speaking, when $\nu_i(p)+\lambda_n$ changes sign for some $i> 1$  then $\nu_1(p)+\lambda_n$ stays negative and therefore the index in the cone $K^n$ is constantly zero.
Next Lemma explains this fact in major details.

Following the notations by Dancer, we introduce 
\begin{align*}
  W_{u_p}:=& \{v\in X^n\ u_p+\gamma v\in K^n\ \text{ for some }\gamma>0\}, \\
S_{u_p}:=& \{v\in W_{u_p}\ : -v\in  W_{u_p}\},
\end{align*}
and denote by $\tilde V$ the orthogonal complement to $S_{u_p}$ in $X^n$ (here orthogonality is meant in $H^1_0$) and by $\tilde T'$ the restriction of $T'$ to $\tilde V$.
It has to be remarked that, as $u_p$ is radial and so does not depend by the angle $\varphi$, then $v\in S_{u_p}$ if and only if {\emph{$v$ does not depend on $\varphi$}} either. Moreover any function which is not constant w.r.t. the angle $\varphi$ (but periodic) belongs to $\tilde V$.

\begin{lemma}\label{soloi=1}
 Let $u_p$ be an isolated fixed point for the map $T$ (i.e. $I-T'$ is invertible), then 
\[
  \mathrm{index}_{K^n}(T( p,\cdot),u_p)=\left\{\begin{array}{ll}
0 & \text{ if }\nu_1(p)+\lambda_n<0 , \\[.1cm]
  \mathrm{index}_{X^n}(T( p,\cdot),u_p) = & \\
{\mathrm{deg}}_{X^n}(I-T(p,\cdot),{\mathcal O}_{p},0)=\pm 1
& \text{ if }\nu_1(p)+\lambda_n>0. \end{array}\right.
\]
\end{lemma}
\begin{proof}
Theorem 1 in \cite{D83} states that
\[
\mathrm{index}_{K^n}(T( p,\cdot),u_p)= \begin{cases}
  0 &   \hbox{ if } T'(p,\cdot)  \mbox{ has the property $\a$}, 
  \\
 \mathrm{index}_{X^n}(T( p,\cdot),u_p) = & \\
{\mathrm{deg}}_{X^n}(I-T(p,\cdot),{\mathcal O}_{p},0)=\pm 1 & \mbox{ otherwise.}
\end{cases}\]
It is therefore needed to read the so called ``property $\alpha$'' in relation with the weighted eigenvalue problem \eqref{eq:linauto}, precisely to show that it holds if and only $\nu_1(p)+\lambda_n<0$.
\\
Several characterizations of the so called ``property $\alpha$'' are provided in Lemma 3 and the following Remark in  \cite{D83}, in particular  $T'$ has the Property $\a$ if 
there exists $t\in (0,1)$ and a $v\in \overline W_{u_p}\setminus S_{u_p}$ which solves
    \[-\Delta v = tp|u_p|^{p-1}v.\]
Following the proof of \cite[Theorem 1]{D92}, we look at the family of eigenvalue problems
\begin{equation}\label{lambdat}
  -\Delta v -t p |u_p|^{p-1}v = \Lambda v \qquad \text{ in } A.
  \end{equation}
For every $t\in[0,1]$, let $\Lambda_t$ be the first eigenvalue for \eqref{lambdat} in $\tilde V$. The variational characterization yields that $\Lambda_t$ is strictly decreasing w.r.t.~$t$.
If $T'$ has the property $\alpha$, then $\Lambda_t\le 0$ for some $t<1$ and therefore $\Lambda_1<0$. This in turn means that \eqref{eq:eigenvalue} has a negative eigenvalue with related eigenfunction in $\tilde V$, and therefore the same must hold for  \eqref{eq:eigenvaluew}, in view of \cite[Lemma 2.1]{GGPS11}.
It has already been remarked that  its eigenvalues  are $\tilde\mu_{i,j}(p)=\nu_i(p)+\lambda_j$ with eigenfunctions described by \eqref{eigenfunction_formula}.
Remembering the description of the eigenfunctions of the Laplace Beltrami operator recalled in the proof of Theorem \ref{teo:bifurcation}, $w_{ij}\in X^n$ implies that $j$ is a multiple of $n$, then $\mu_1(p)+\lambda_n\le \tilde\mu_{i,j}(p)<0$.
\\
On the other hand, if $\mu_1(p)+\lambda_n<0$, then the function
\[ v=\psi_1(r) \prod_{k=1}^{\n-2}G_n^n(\cos \theta_k,k-1) \cos(n\varphi)  = \psi_1(r) \cos(n\varphi)  \prod_{k=1}^{\n-2}\left( \sin \theta_k\right)^n\]
 is an eigenfunction for \eqref{eq:eigenvaluew} related to the eigenvalue $\mu_1(p)+\lambda_n<0$.  Because $v\in\tilde V$,  the eigenvalue $\Lambda_1$ of \eqref{lambdat} is negative, and therefore also $\Lambda_t<0$ for some $t\in(0,1)$, which means that $T'$ has the property $\alpha$.
\end{proof}

Since the change in the Leray-Schauder degree (computed as index of isolated solutions) is the key ingredient in the proof of the bifurcation results then all the standard results can be proved in the cone $K^n$ at the values of the parameter $p$ at which $\nu_1(p)+\l_n$ changes sign. 

\

We are in position to proof our main result

\begin{proof}[Proof of Theorem \ref{teo:bifurcationincones}]
Starting from Lemma \ref{soloi=1} we have that the Leray Schauder degree in the cone $K^n$ changes at the values of $p$ at which 
$\nu_1(p)+\l_n$ changes sign (and this happens due to \eqref{asymptoticnuinfty}, \eqref{asymptoticnu1l} and \eqref{asymptoticnu1m} for $n\geq \bar{n}$ or for any $n\geq 1$ when $m=1$).

\

This change in the index implies a change in the Leray Schauder degree in the cone $K^n$ and this implies the global bifurcation.
Indeed one can reproduce the proof of the Rabinowitz alternative as in \cite{G} Theorem 3.3 from Step 1 (which holds only in the case of positive solutions, i.e. $m=1$) to Step 3. Letting $C_n$ be the closed connected component of $\Sigma^m_n$ that contains $(p_n,u_{p_n})$ then we get that or $C_n$ is unbounded in the product space $(1,+\infty)\times K^n$ or it intersects $\{1\}\times K^n$ or, following Step 4 in Theorem 3.3 in \cite{G}, it is bounded in  $(1,+\infty)\times K^n$ and it must intersect the curve of radial solutions $S^m$  in another point at which $\nu_1(p)+\l_n$ changes sign. This follows since it has to be a point at which the Leray Schauder degree, and hence the index of fixed point in the cone,  changes. 

\

Finally we can apply Proposition 3.2 of \cite{G16} and, if $C_n$ is bounded then the number of points on $C_n$ at which $\nu_1(p)+\l_n$ changes sign has to be even. Since $\nu_1(p)+\l_n>0$ for $p$ close to $1$ while $\nu_1(p)+\l_n<0$ for $p$ large enough then the points at which $\nu_1(p)+\l_n$ changes sign are odd and so there is at least one of these points at which alternative $i)$, $ii)$ or $iii)$ holds. 
Finally when case $iii)$ occurs, i.e. the component $C_n$ intersects $\{1\}\times K^n$, then \cite[Theorem 1.4]{Gro09} applies and item iii) follows concluding the proof.

\end{proof}

\end{document}